\documentclass[10pt,a4paper]{amsart}
\usepackage{verbatim}

\usepackage[toc]{appendix}
\usepackage[T1]{fontenc}
\usepackage{graphicx}
\usepackage{enumerate}
\usepackage{amsmath,amsfonts,amssymb}
\usepackage{color}
\usepackage{cite}
\definecolor{citation}{rgb}{0.2,0.58,0.2} 
\definecolor{formula}{rgb}{0.1,0.2,0.6}
\definecolor{url}{rgb}{0.3,0,0.5}
\usepackage{pgf,tikz}
\usepackage{mathrsfs}
\usepackage{fancyhdr}
\usepackage{dutchcal}

\usepackage{dsfont}

\newcommand{\reqnomode}{\tagsleft@false}

\usepackage[colorlinks,pdfpagelabels,pdfstartview = FitH,bookmarksopen = true,bookmarksnumbered = true,urlcolor=url,linkcolor = formula,plainpages = false,hypertexnames = false,citecolor = citation] {hyperref}

\newcommand{\cd}[1]{\textcolor[rgb]{0.00,0.00,0.00}{#1}}

\allowdisplaybreaks
\makeatletter
\DeclareRobustCommand*{\bfseries}{%
  \not@math@alphabet\bfseries\mathbf
  \fontseries\bfdefault\selectfont
  \boldmath
}

\makeatother

\newlength{\defbaselineskip}
\newcommand{\setlinespacing}[1]
           {\setlength{\baselineskip}{#1 \defbaselineskip}}

\newtheorem{thrm}{Theorem}[section]
\newtheorem{lmm}{Lemma}[section]

\numberwithin{equation}{section}
\allowdisplaybreaks[4]

\newcommand\eps\varepsilon
\def\eqn#1$$#2$${\begin{equation}\label#1#2\end{equation}}

\newcommand{\be}{\begin{equation}}
\newcommand{\ee}{\end{equation}}

\newcommand{\snr}[1]{\lvert #1\rvert}

\def\name[#1, #2]{#1 #2}

\title[Uniform ellipticity and $p$-$q$ growth]{Uniform ellipticity and $p$-$q$ growth}

\author[De Filippis]{Cristiana De Filippis}  \address{Cristiana De Filippis\\Mathematical Institute, University of Oxford\\ Andrew Wiles Building, Radcliffe Observatory Quarter, Woodstock Road, Oxford, OX26GG, Oxford, United Kingdom} \email{\texttt{Cristiana.DeFilippis@maths.ox.ac.uk}}
\author[Leonetti]{Francesco Leonetti}  \address{Francesco Leonetti\\Universit\`a di L'Aquila, 
 Dipartimento di Ingegneria e Scienze dell'Informazione e Matematica\\ Via Vetoio snc, 67100 L'Aquila, Italy} \email{\texttt{leonetti@univaq.it}}
\begin{document}

\subjclass[2010]{35J60, 35J70\vspace{1mm}} 

\keywords{Regularity, uniform ellipticity, $p$-$q$-growth\vspace{1mm}}

\thanks{{\it Acknowledgements.}\ C. De Filippis is supported by the Engineering and Physical Sciences Research Council (EPSRC): CDT Grant Ref. EP/L015811/1. F. Leonetti is supported by MIUR, GNAMPA, INdAM, UNIVAQ. 
\vspace{1mm}}

\maketitle

\begin{abstract}
Fix any two numbers $p$ and $q$, with $1<p<q$; we give an example of an integral functional enjoying uniform ellipticity and $p$-$q$ growth. 
\end{abstract}
\vspace{3mm}
{\small \tableofcontents}

\setlinespacing{1.08}

\section{Introduction}

\noindent
We consider integral functionals
\begin{equation}
\label{integral_functional}
\int_{\Omega} f(Du(x)) dx,
\end{equation}
\noindent
where $u:\Omega \subset \mathbb{R}^n \to \mathbb{R}^N$, $\Omega$ is bounded and open and $f$ is continuous and nonnegative.
About $f$ we assume $p$-$q$ growth
\begin{equation}
\label{p-q_growth}
c_1 |z|^p - c_2 \leq f(z) \leq c_3 |z|^q + c_4,
\end{equation}

\noindent
where $c_1, c_2, c_3, c_4, p, q$ are constants with $c_1, c_3 \in (0, +\infty)$, $c_2, c_4 \in [0, +\infty)$ and $1 < p < q$.
In this framework it is usual to assume that

\begin{equation}
\label{p-2_growth}
c_5 (\mu + |z|)^{p-2} \leq \langle DDf(z) \frac{\lambda}{|\lambda|}, \frac{\lambda}{|\lambda|} \rangle,
\end{equation}

\noindent
and

\begin{equation}
\label{q-2_growth}
|DDf(z)| \leq c_6 (\mu + |z|)^{q-2},
\end{equation}

\noindent
where $c_5, c_6, \mu$ are constants with $c_5, c_6 \in (0, +\infty)$ and \cd{$\mu\in [0,1]$}%$\mu \in [0, +\infty)$
.
\cd{Retaining only the informations about the growth in the large of the second derivative, as prescribed by (\ref{p-2_growth})-(\ref{q-2_growth}), leads to the following bound on the ratio between the highest and the lower eigenvalue of $DDf$:}
%This implies that the ratio of the highest and the lower eigenvalues of $DDf(z)$ can be estimated as follows
\begin{equation}
\label{ratio}
\cd{\mathcal{R}(z):=}\frac{\text{highest eigenvalue of }DDf(z)}{\text{lowest eigenvalue of }DDf(z)}
\leq c_7 (\mu + |z|)^{q-p},
\end{equation}

\noindent
for some positive constant $c_7$. %if we look at 
\cd{The right hand side of (\ref{ratio}),} \cd{evidently blows up as $\snr{z}\to \infty$, given that, in general, $q>p$. On the other hand, if by any chance the integrand $f$ features certain structural properties which make $\mathcal{R}(z)$ bounded from above by a constant non-depending, in particular, from $z$, then we have uniform ellipticity.} %we see that we loose uniform ellipticity when $z$ is large. Note that the left hand side of (\ref{ratio}) might stay bounded; in this last case we have uniform ellipticity.
We are concerned with regularity of minimizers $u:\Omega \subset \mathbb{R}^n \to \mathbb{R}^N$ of (\ref{integral_functional}); in this framework of $p$-$q$ growth, the following bound sometimes appears

\begin{equation}
\label{q_close_to_p}
q < p + c(n,p),
\end{equation}

\noindent
where $c(n,p)$ is positive and tends to $0$ when the dimension $n$ tends to $+\infty$; see \cite{mar1989, Moscariello_Nania1991, Fusco_Sbordone1993, Esposito_Leonetti_Mingione1999, Esposito_Leonetti_Mingione2004, Baroni_Colombo_Mingione2015, Eleuteri_Marcellini_Mascolo2016, Cupini_Marcellini_Mascolo2017} and \cite[Section 6]{Mingione2006}; see also \cite[Section 6.2]{DeFilippis-Mingione} where a simple argument is given.  
Now we assume the following structure condition
\begin{equation}
\label{structure}
f(z) = g(|z|),
\end{equation}
\noindent 
with $g:[0, +\infty) \to [0, +\infty)$. 
Some papers require $g(0)=0$, $g \in C^2((0,+\infty)) \cap C^1([0,+\infty))$ with  $g'(t)>0$ for $t>0$; moreover,
\cite{Lieberman, Diening_Ettwein, Cianchi_Maz_ya_Archive2014, Beck_Mingione} ask for
% and $g(t)>0$ for $t>0$; 
%moreover 
%\begin{equation*}
%\lim\limits_{t \to +\infty} \frac{g(t)}{t} = +\infty,
%\qquad
%\lim\limits_{t \to 0} \frac{g(t)}{t} = 0.
%\end{equation*}
\begin{equation}
\label{assumption_BM}
0 < m \leq \frac{g''(t) t}{g'(t)} \leq M < +\infty
\qquad
\forall t > 0.
\end{equation}
Note that \cite{Baroni} requires \eqref{assumption_BM} with $1 \leq m$; on the other hand, \cite{Leonetti_Mascolo_Siepe2003} asks for 
$M \leq 1$.  
In \cite{Bildhauer_Fuchs} they consider splitting densities $f(Du) = a(|(D_1 u,...,D_{n-1} u)|) + b(|D_n u|)$ and they require  
\eqref{assumption_BM} for both $a$ and $b$. 
We remark that $g'>0$ and \eqref{assumption_BM} forces $g''>0$, so $g$ must be strictly convex; on the other hand, \eqref{assumption_BM} allows $p$-$q$ growth whatever $p$ and $q$ are: in this paper we fix $p$ and $q$ with $1<p<q$, no matter how far they are, and we show a convex function $g$ verifing \eqref{assumption_BM}, with $p$-$q$ growth. 
In \cite{Beck_Mingione} we find Theorem 1.15 that says
\begin{thrm}
\label{thm_BM}
We assume that $g(0)=0$ and $g \in C^2((0,+\infty)) \cap C^1([0,+\infty))$; moreover,  $g'(t)>0$ for $t>0$ and \eqref{assumption_BM} holds true.
% and $g(t)>0$ for $t>0$; 
%moreover 
%\begin{equation*}
%\lim\limits_{t \to +\infty} \frac{g(t)}{t} = +\infty,
%\qquad
%\lim\limits_{t \to 0} \frac{g(t)}{t} = 0.
%\end{equation*}
%\begin{equation}
%\label{thm_assumption_BM}
%0 < m \leq \frac{g''(t) t}{g'(t)} \leq M < +\infty
%\qquad
%\forall t > 0.
%\end{equation}
If $u \in W^{1,1}_{loc}(\Omega,\mathbb{R}^N)$ is a local minimizer of (\ref{integral_functional}) under the structure condition 
(\ref{structure}) with $g$ as before, then $u$ is locally Lipschitz continuous in $\Omega$.
\end{thrm}
% moreover, in Theorem \ref{thm_LMS2003}, the restriction (\ref{q_close_to_p}) does not appear: $p$ and $q$ are free between $1$ and $2$.
\noindent 
We are going to show an example for the previous Theorem \ref{thm_BM}: fix $p$ and $q$ with 
$1<p<q$, then we give $g$ satisfying all the assumptions of Theorem \ref{thm_BM} with the chosen $p$ and $q$: the restriction (\ref{q_close_to_p}) does not apply! Moreover, such a $g$ gives   an $f$ for which we have uniform ellipticity; indeed, let $g$ be any function in $C^2((0,+\infty))$ with $g'(t)>0$ for $t>0$, satisfying assumption 
\eqref{assumption_BM}; 
%(\ref{senza_nome_LMS_2003}) of Theorem \ref{thm_LMS2003};
 then, for the corresponding $f$ given by (\ref{structure}), we have
\begin{equation*}
\frac{\partial f}{\partial z^{\alpha}_{i}} (z) = g^{\prime}(|z|) \frac{z^{\alpha}_{i}}{|z|}
\end{equation*}
and
\begin{equation*}
\frac{\partial^2 f}{\partial z^{\alpha}_{i}\partial z^{\beta}_{j}} (z) = 
\left[
g^{\prime\prime}(|z|) - \frac{g^{\prime}(|z|)}{|z|} 
\right]\cd{\frac{z^{\alpha}_{i}z^{\beta}_{j}}{\snr{z}^{2}}}
%\frac{z^{\alpha}_{i}}{|z|}
%\frac{z^{\beta}_{j}}{|z|}
+
\frac{g^{\prime}(|z|)}{|z|}
\delta^{\alpha\beta}
\delta_{ij},
\end{equation*}
so that
\begin{equation}
\label{quadratic_form}
\left\langle 
DDf(z) \frac{\lambda}{|\lambda|}, \frac{\lambda}{|\lambda|}
\right\rangle
=
\left[
g^{\prime\prime}(|z|) - \frac{g^{\prime}(|z|)}{|z|} 
\right]
\left\langle
\frac{z}{|z|}, \frac{\lambda}{|\lambda|}
\right\rangle^2
+
\frac{g^{\prime}(|z|)}{|z|};
\end{equation}
if we consider first the case $[...] \geq 0$ 
%$g^{\prime\prime}(|z|) - \frac{g^{\prime}(|z|)}{|z|} \geq 0$ 
and then the other case
$[...] < 0$, using \eqref{assumption_BM}, 
%$g^{\prime\prime}(|z|) - \frac{g^{\prime}(|z|)}{|z|} < 0$, 
we get
%the left hand side of  
%(\ref{senza_nome_LMS_2003}) says that $[...]\leq 0$ so that
%\begin{equation}
%\label{quadratic_form_2}
%\langle 
%DDf(z) \frac{\lambda}{|\lambda|}, \frac{\lambda}{|\lambda|}
%\rangle
%\leq 
%\frac{g^{\prime}(|z|)}{|z|};
%\end{equation}
%the right hand side of  
%(\ref{senza_nome_LMS_2003}) says that $\left(\frac{1}{\gamma}-1\right) \frac{g^{\prime}(|z|)}{|z|} \leq [...]$ so that
%\begin{equation}
%\label{quadratic_form_3}
%\frac{1}{\gamma}
%\frac{g^{\prime}(|z|)}{|z|}
%\leq
%\langle 
%DDf(z) \frac{\lambda}{|\lambda|}, \frac{\lambda}{|\lambda|}
%\rangle;
%\end{equation}
%then
\begin{equation}
\label{ratio_2}
\frac{\text{highest eigenvalue of }DDf(z)}{\text{lowest eigenvalue of }DDf(z)}
\leq
\max 
\left\{M; \frac{1}{m} \right\};
\end{equation}
so, we are in the uniform ellipticity regime. So, after fixing $p$ and $q$ at will in $(1, +\infty)$, we are going to write an example of functional with $p$-$q$ growth and unifom ellipticity.
For $1<p<q$, set $a=\frac{p+q}{2}$ and $b=\frac{q-p}{2}$. Then, we have $a,b >0$, $1 <p = a-b < a+b = q$ and we can use the function $g$ defined in the next section \ref{example}.

\section{Example}
\label{example}

\noindent
We fix $a,b \in (0, +\infty)$ with
\begin{equation}
\label{conditions_for_a_b}
1<a-b.
\end{equation}
\noindent
We consider $g:[0, +\infty) \to [0, +\infty)$ such that
\begin{equation}
\label{definition_for_g}
g(t) = t^{a+b\sin(\varphi(t))},
\end{equation}
\noindent
where $\varphi: \mathbb{R} \to \mathbb{R}$ is 
given by
\begin{equation}
\label{definition_for_varphi}
\varphi(t)
=
\left\{
\begin{array}{ll}
\frac{3}{2} \pi & \text{ if } t \in (-\infty, 1],
\\
  &  
\\
\frac{3}{2} \pi  + \varepsilon \ln \ln (e + (t-1)^4) & \text{ if } t \in (1, +\infty),
\end{array}
\right.
\end{equation}
\noindent
with $\varepsilon >0$. 
%to be choosen in the next lines. 
Note that 
\begin{equation}
\label{varphi_prime}
\varphi^{\prime}(t)
=
\left\{
\begin{array}{ll}
0 & \text{ if } t \in (-\infty, 1],
\\
  &  
\\
 \frac{\varepsilon}{\ln (e + (t-1)^4)} \, \, \frac{4 (t-1)^3}{e + (t-1)^4} & \text{ if } t \in (1, +\infty)
\end{array}
\right.
\end{equation}
\noindent
and
\begin{equation}
\label{varphi_prime_prime}
\varphi^{\prime\prime}(t)
=
\left\{
\begin{array}{ll}
0 &  \text{ if } t \in (-\infty, 1],
\\
  &  
\\
\varepsilon
\left\{
 \frac{-1}{[\ln (e + (t-1)^4)]^2} \left[ \frac{4 (t-1)^3}{e + (t-1)^4}\right]^2 +
\right.
& 
\\ 
\phantom{aaaaaaaaaa}
\left.
\frac{1}{\ln (e + (t-1)^4)} \, \, \frac{12(t-1)^2 e - 4 (t-1)^6}{[e + (t-1)^4]^2}
\right\}
 & \text{ if } t \in (1, +\infty)
\end{array}
\right.
\end{equation}
\noindent
so that $\varphi \in C^2(\mathbb{R})$. Note that $\varphi^{\prime}(t)>0$ when $t>1$; moreover,  
$\lim\limits_{t \to +\infty} \varphi(t) = +\infty$. Then $\varphi(t)$ increases and takes all the values of the interval $[\frac{3}{2} \pi, +\infty)$. This means that, in (\ref{definition_for_g}), the exponent $a+b\sin(\varphi(t))$ oscillates between $a-b$ and $a+b$ infinitely many times as $t$ goes from $0$ to $+\infty$; then $g(t)$ has $a-b$ growth from below and $a+b$ growth from above.
As far as $\varepsilon$ is concerned, we require that 
\begin{equation}
\label{epsilon_1}
0 < \varepsilon < 
\min
\left\{
1;
\,\, 
\frac{a-1-b}{224\,\,b}
%;
%\,\,
%\frac{2-a-b}{224\,\,b}
\right\}.
\end{equation}

\noindent
We are going to prove the next

\begin{thrm}
\label{thm_g}
Let us consider $a,b \in (0, +\infty)$ verifing (\ref{conditions_for_a_b}); we take $g(t)$ given by (\ref{definition_for_g}) where $\varphi$ is defined in (\ref{definition_for_varphi}) and $\varepsilon$ satisfies (\ref{epsilon_1}). Then $g:[0, +\infty) \to [0, +\infty)$, $g(0)=0$, $g(t) > 0$ for $t>0$, $g \in C^1([0, +\infty)) \cap C^2((0, +\infty))$, $\lim\limits_{t \to 0^+}\frac{g(t)}{t} = 0$, $\lim\limits_{t \to +\infty}\frac{g(t)}{t} = +\infty$, $g^{\prime}(0) = 0$ and  
\begin{equation}
\label{statement_1}
0 < \{ -b8\varepsilon + a - b \} \frac{g(t)}{t} \leq g^{\prime}(t) \leq \{ b8\varepsilon + a + b \} \frac{g(t)}{t}
%\qquad \forall t > 0;
\end{equation}
for every $t>0$; 
moreover, 
\begin{equation}
\label{statement_2}
0 <
\{-b 8 \varepsilon + a - b\}
t^{a-b-1} 
\leq
g^{\prime}(t)
\leq
\{b 8 \varepsilon + a + b\}
[t^{a-b-1} + t^{a+b-1}]
%\qquad
%\forall t > 0.
\end{equation}
for every $t>0$. 
As far as $g^{\prime\prime}$ is concerned, we get
\begin{equation}
\label{statement_3}
0 < \{ -224b\varepsilon + a - 1 - b \} \frac{g^{\prime}(t)}{t} \leq g^{\prime\prime}(t) \leq \{ 224 b\varepsilon + a - 1 + b \} \frac{g^{\prime}(t)}{t}
%\qquad \forall t > 0,
\end{equation}
for every $t>0$, 
thus $g$ is strictly convex in $[0, +\infty)$.

%Let $g(t) = t^{a+b\sin(\varphi(t))}$ with $\varphi$ as before.  
%Then $g$ verifies all the assumptions of Theorem \ref{thm_LMS2003}; more precisely, (\ref{1.2_LMS_2003}) holds true with $p=a-b$, $q=a+b$, $\Lambda_1=-b8\varepsilon + a - b$, $\Lambda_2=b8\varepsilon + a + b$ and (\ref{senza_nome_LMS_2003}) is satisfied with $\gamma = \frac{1}{a-1-b-224b\varepsilon}$.
\end{thrm}

%\noindent
%Note that (\ref{epsilon_1}) implies $\Lambda_1=-b8\varepsilon + a - b > 1$ and

%\begin{eqnarray}
%\label{comparison_1}
%1>1-b>a-1-b>a-1-b-224 b \varepsilon >0,
%\end{eqnarray}
 
%\noindent
%then, 

%\begin{eqnarray}
%\label{gamma>1}
%\gamma = 
%\frac{1}
%{a-1-b-224b\varepsilon } >1
%\end{eqnarray}
 
\noindent
%as required. 
The present example is a modification of the one given in \cite{Fusco_Sbordone1990,Talenti1990}; in the present example the small new parameter $\varepsilon$ appears and it makes possible to get convexity and $p$-$q$ growth with any $p$ and $q$ 
%between $1$ and $2$, as well the estimates required in Theorem \ref{thm_LMS2003}
.

\section{Preliminary results}\label{pre}

\noindent
We need some preliminary estimates.

\begin{lmm}
\label{estimate_for_ratio_1} \cd{For all $t\in (1,\infty)$ there holds that:}
\begin{equation}
\label{inequality_for_ratio_1}
0 < \frac{(t-1)^3}{e+(t-1)^4} < 1.
%\qquad
%\forall t \in (1, +\infty).
\end{equation}
\end{lmm}

\begin{proof}
If $1<t\leq 2$, then $0<t-1\leq 1$ so that
$$
0<\frac{(t-1)^3}{e+(t-1)^4} \leq \frac{1}{e+(t-1)^4} \leq \frac{1}{e} < 1.
$$
If $2<t$, then $1<t-1$ so that
$$
0<\frac{(t-1)^3}{e+(t-1)^4} <
\frac{(t-1)^4}{e+(t-1)^4}
 < 1.
$$
\noindent
The two cases give (\ref{inequality_for_ratio_1}). 
%This ends the proof. 
\end{proof}

\begin{lmm}
\label{estimate_for_ratio_2}\cd{For all $t\in (1,\infty)$ there holds that:}
\begin{equation}
\label{inequality_for_ratio_2}
0 < \frac{(t-1)^3t}{e+(t-1)^4} < 2.
%\qquad
%\forall t \in (1, +\infty).
\end{equation}
\end{lmm}

\begin{proof}
We write $t=(t-1) + 1$ and we get
$$
0<\frac{(t-1)^3t}{e+(t-1)^4} 
=
\frac{(t-1)^3(t-1)}{e+(t-1)^4} + \frac{(t-1)^3}{e+(t-1)^4} < 
\frac{(t-1)^4}{e+(t-1)^4} + 1 < 1+1,
$$
where we used (\ref{inequality_for_ratio_1}).
%This ends the proof. 
\end{proof}

\begin{lmm}
\label{estimate_for_ratio_3}\cd{For all $t\in (1,\infty)$ there holds that:}
\begin{equation}
\label{inequality_for_ratio_3}
0 < \frac{(t-1)^2 t^2}{e+(t-1)^4} < 4.
%\qquad
%\forall t \in (1, +\infty).
\end{equation}
\end{lmm}

\begin{proof}
If $1<t<2$, then $0<t-1<1$ so that
$$
0<\frac{(t-1)^2 t^2}{e+(t-1)^4} <
\frac{4}{e+(t-1)^4}  < \frac{4}{e} < \frac{4}{2} = 2.
$$
If $2 \leq t$, then $t \leq 2(t-1)$ so that
$$
0<\frac{(t-1)^2 t^2}{e+(t-1)^4} \leq \frac{(t-1)^2 4(t-1)^2}{e+(t-1)^4} = \frac{4(t-1)^4 }{e+(t-1)^4} < 4.
$$
\noindent
The two cases give (\ref{inequality_for_ratio_3}).
\end{proof}

\begin{lmm}
\label{estimate_for_ratio_4}\cd{For all $t\in (1,\infty)$ there holds that:}
\begin{equation}
\label{inequality_for_ratio_4}
0 < \frac{\ln t}{\ln(e+(t-1)^4)} < 1.
%\qquad
%\forall t \in (1, +\infty).
\end{equation}
\end{lmm}

\begin{proof}
If $1<t \leq e$, then
$$
0 < \frac{\ln t}{\ln(e+(t-1)^4)} \leq \frac{\ln e}{\ln(e+(t-1)^4)} < \frac{\ln e}{\ln e} = 1.
$$
If $e<t$, then $t<(t-1)^2$: indeed, this last inequality is equivalent to $0<t^2-3t+1$; the two solutions of the equation $t^2-3t+1=0$ are 
$\frac{3-\sqrt{5}}{2}$ and $\frac{3+\sqrt{5}}{2}$; note that $5<5,29=(2,3)^2$, so that $\frac{3+\sqrt{5}}{2}< \frac{3+2,3}{2} = 2,65 < e$; then $e<t$ implies $0<t^2-3t+1$ and $t<(t-1)^2$. This last inequality allows us to write
$$
0 < \frac{\ln t}{\ln(e+(t-1)^4)} < \frac{\ln ((t-1)^2)}{\ln(e+(t-1)^4)} < \frac{\ln ((t-1)^2)}{\ln((t-1)^4)} = \frac{2}{4}.
$$
\noindent
The two cases give (\ref{inequality_for_ratio_4}).
\end{proof}

\noindent
In this section \ref{pre}, $\varphi$ is given by (\ref{definition_for_varphi}) with any $\varepsilon > 0$: \cd{in the forthcoming lemmas,} no restriction from above on 
$\varepsilon$ is required.

\begin{lmm}
\label{estimate_for_varphi_prime_1}\cd{Let $\varepsilon>0$ be any number and $\varphi$ be the function in (\ref{definition_for_varphi}). Then,}
\begin{equation}
\label{inequality_for_varphi_prime_1}
0 < \varphi^{\prime}(t) t \ln t \leq 8 \varepsilon
\qquad
\forall t \in (1, +\infty).
\end{equation}
\end{lmm}

\begin{proof}
We take into account formula (\ref{varphi_prime}) and estimates (\ref{inequality_for_ratio_4}), (\ref{inequality_for_ratio_2}):
\begin{eqnarray*}
0 < 
\varphi^{\prime}(t) t \ln t = \frac{\varepsilon}{\ln (e + (t-1)^4)} \, \, \frac{4 (t-1)^3}{e + (t-1)^4} t \ln t = 
\\
\varepsilon \frac{\ln t}{\ln (e + (t-1)^4)} \, \, \frac{4 (t-1)^3 t}{e + (t-1)^4} \leq \varepsilon 8.
\end{eqnarray*}

\end{proof}

\begin{lmm}
\label{estimate_for_varphi_prime_2}\cd{Let $\varepsilon>0$ be any number and $\varphi$ be the function in (\ref{definition_for_varphi}). Then,}
\begin{equation}
\label{inequality_for_varphi_prime_2}
0 < \varphi^{\prime}(t) t  \leq 8 \varepsilon
\qquad
\forall t \in (1, +\infty).
\end{equation}
\end{lmm}

\begin{proof}
We take into account formula (\ref{varphi_prime}) and estimate (\ref{inequality_for_ratio_2}):
\begin{eqnarray*}
0 < 
\varphi^{\prime}(t) t  = \frac{\varepsilon}{\ln (e + (t-1)^4)} \, \, \frac{4 (t-1)^3}{e + (t-1)^4} t  = 
\\
\varepsilon \frac{1}{\ln (e + (t-1)^4)} \, \, \frac{4 (t-1)^3 t}{e + (t-1)^4} \leq \varepsilon 8.
\end{eqnarray*}

\end{proof}

\begin{lmm}
\label{estimate_for_varphi_prime_prime}\cd{Let $\varepsilon>0$ be any number and $\varphi$ be the function in (\ref{definition_for_varphi}). Then,}
\begin{equation}
\label{inequality_for_varphi_prime_prime}
|\varphi^{\prime\prime}(t)| t^2 \ln t  \leq 128 \varepsilon
\qquad
\forall t \in (1, +\infty).
\end{equation}
\end{lmm}

\begin{proof}
We take into account formula (\ref{varphi_prime_prime}) and estimates (\ref{inequality_for_ratio_2}), (\ref{inequality_for_ratio_3}), 
(\ref{inequality_for_ratio_4}):
\begin{eqnarray*}
|\varphi^{\prime\prime}(t)| t^2 \ln t \leq 
\frac{\varepsilon \ln t}{[\ln (e + (t-1)^4)]^2} \left[ \frac{4 (t-1)^3t}{e + (t-1)^4}\right]^2 +
\phantom{aaaaaaaaaaaaaaaaa}
\\
\frac{\varepsilon \ln t}{\ln (e + (t-1)^4)} \, \, \frac{12 e (t-1)^2 t^2  + 4 (t-1)^6 t^2}{[e + (t-1)^4]^2}
\leq 
\varepsilon (4^3 + 48 + 16) = 128 \varepsilon.
\end{eqnarray*}

\end{proof}

\section{Proof of Theorem \ref{thm_g}}

\noindent
Definitions (\ref{definition_for_g}) and (\ref{definition_for_varphi}) say that, when $t \in [0, 1]$, 
$\varphi(t)=\frac{3}{2}\pi$ and 
$g(t) = t^{a-b}$; condition (\ref{conditions_for_a_b}) guarantees that $1<a-b$ so that 
\begin{equation}
\label{g(0)}
g(0)=0,
\end{equation}

\begin{equation}
\label{g(t)/t_at_0}
\lim\limits_{t \to 0^+}\frac{g(t)}{t}=0,
\end{equation}

\begin{equation}
\label{g^prime(0)}
g^{\prime}(0)=0;
\end{equation}

\noindent
moreover, $g(t)>0$ for $t>0$. We recall that, for $t>1$, $t^{a-b} \leq g(t)$; again, 
condition (\ref{conditions_for_a_b}) guarantees that $1<a-b$ so that

\begin{equation}
\label{g(t)/t_at_infinity}
\lim\limits_{t \to +\infty}\frac{g(t)}{t}=+\infty.
\end{equation}

\noindent
Up to now, $g \in C^0([0, +\infty))$.
For $t>0$ we have

\begin{equation}
\label{esponential_formula_for_g}
g(t) = t^{a+b\sin(\varphi(t))} = e^{[a+b\sin(\varphi(t))]\ln t},
\end{equation}

\noindent
so that

\begin{eqnarray}
\label{g^prime}
g^{\prime}(t) = e^{[a+b\sin(\varphi(t))]\ln t} 
\left\{
[b \cos(\varphi(t))] \varphi^{\prime}(t) \ln t + [a+b\sin(\varphi(t))] \frac{1}{t}
\right\} =
\nonumber
\\
\frac{g(t)}{t}
\left\{
[b \cos(\varphi(t))] \varphi^{\prime}(t) t \ln t + [a+b\sin(\varphi(t))]
\right\}.
\end{eqnarray}

\noindent
If $t \in (0, 1]$, then $\varphi(t) = \frac{3}{2} \pi$ and $\varphi^{\prime}(t) = 0$, so that

\begin{equation}
\label{g_prime_in_(0_1]}
g^{\prime}(t) = \frac{g(t)}{t} [a-b] = [a-b] t^{a-b-1};
\end{equation}

\noindent
again, 
condition (\ref{conditions_for_a_b}) guarantees that $1<a-b$ so that

\begin{equation}
\label{lim_g_prime}
\lim\limits_{t \to 0+} g^{\prime}(t) = 0.
\end{equation}

\noindent
Then, 
%continuity of $g$ in $0$ and de L'Hospital's rule give 
%\begin{equation}
%\label{g_prime(0)}
%g^{\prime}(0) = 0;
%\end{equation}
%\noindent
%moreover, 
$g \in C^1([0, +\infty))$. Using formula (\ref{g^prime}), when $t>0$, we have

\begin{eqnarray}
\label{g^{prime_prime}}
g^{\prime\prime}(t) = 
\frac{g^{\prime}(t) t - g(t)}{t^2}
\left\{
[b \cos(\varphi(t))] \varphi^{\prime}(t) t \ln t + [a+b\sin(\varphi(t))]
\right\} +
\phantom{aaaaa}
\nonumber
\\
\frac{g(t)}{t}
\left\{
[-b \sin(\varphi(t))] \varphi^{\prime}(t) \varphi^{\prime}(t) t \ln t + 
\phantom{aaaaaaaaaaaaaaaaaaaaaaa}
\right.
\nonumber
\\
\left.
[b \cos(\varphi(t))] [\varphi^{\prime\prime}(t) t \ln t + \varphi^{\prime}(t) (\ln t + 1)] + 
[b \cos(\varphi(t))] \varphi^{\prime}(t)
\right\}.
\end{eqnarray}

\noindent
Then $g \in C^2((0, +\infty))$. 
Now we are going to estimate 
$g^{\prime}(t)$ by means of $\frac{g(t)}{t}$.
First of all, we consider the case $t \in (0, 1]$: we can use formula 
(\ref{g_prime_in_(0_1]}) and we get $g^{\prime}(t) = (a-b) \frac{g(t)}{t}$. After that, we deal with $t>1$; 
we use formula (\ref{g^prime}) and estimate (\ref{inequality_for_varphi_prime_1}):

\begin{eqnarray}
\label{g_prime_estimate_1}
\frac{g(t)}{t}
\{-b 8 \varepsilon + a - b\} 
\leq
\phantom{aaaaaaaaaaaaaaaaa}
\nonumber
\\
\underbrace{
\frac{g(t)}{t}
\left\{
[b \cos(\varphi(t))] \varphi^{\prime}(t) t \ln t + [a+b\sin(\varphi(t))]
\right\}}_{g^{\prime}(t)}
\leq
\phantom{aaaaaa}
\nonumber
\\
\frac{g(t)}{t}
\{b 8 \varepsilon + a + b\}. 
\end{eqnarray}

\noindent
Note that $-b 8 \varepsilon + a - b < a-b < a+b < b 8 \varepsilon + a + b$; then

\begin{equation}
\label{g_prime_estimate_2}
\frac{g(t)}{t}
\{-b 8 \varepsilon + a - b\} 
\leq
g^{\prime}(t)
\leq
\frac{g(t)}{t}
\{b 8 \varepsilon + a + b\}
\qquad
\forall t > 0.
\end{equation}

\noindent
Up to now, we only used $a,b>0$, $1<a-b$ and $\varepsilon > 0$. 
Assumption (\ref{epsilon_1}) 
%$\boxed{0 < \varepsilon < \frac{a-b}{8b}}$ 
guarantees that
$\varepsilon < \frac{a-1-b}{224 b}$; then $8b\varepsilon < 224 b \varepsilon < a-1-b$, so that
$1<-b 8 \varepsilon + a - b$; this and positivity of $g$ give $g^{\prime}(t) > 0$ when $t>0$. 
Moreover, (\ref{g_prime_estimate_2}) can be written as follows

\begin{equation}
\label{g_prime_estimate_3}
\frac{1}{b 8 \varepsilon + a + b}
g^{\prime}(t)
\leq
\frac{g(t)}{t} 
\leq
\frac{1}{-b 8 \varepsilon + a - b}
g^{\prime}(t)
\qquad
\forall t > 0.
\end{equation}

\noindent
We note that 
 
\begin{equation}
\label{g_in_(1,infinity)}
t^{a-b} \leq  \underbrace{t^{a+b\sin(\varphi(t))}}_{g(t)} \leq t^{a+b}
\qquad
\forall t>1;
\end{equation}

\noindent
then we use estimates (\ref{g_prime_estimate_2}), (\ref{g_in_(1,infinity)}) and positivity of $-b 8 \varepsilon + a - b$:

\begin{equation}
\label{g_prime_estimate_4}
t^{a-b-1}
\{-b 8 \varepsilon + a - b\} 
\leq
g^{\prime}(t)
\leq
t^{a+b-1}
\{b 8 \varepsilon + a + b\}
\qquad
\forall t > 1.
\end{equation}

\noindent
We keep in mind that $g^{\prime}(t) = (a-b) \frac{g(t)}{t} = (a-b) t^{a-b-1}$ for $t \in (0, 1]$; moreover, $-b 8 \varepsilon + a - b < a-b < a+b < b 8 \varepsilon + a + b$; then 

\begin{equation}
\label{g_prime_estimate_5}
t^{a-b-1}
\{-b 8 \varepsilon + a - b\} 
\leq
g^{\prime}(t)
\leq
[t^{a-b-1} + t^{a+b-1}]
\{b 8 \varepsilon + a + b\}
\qquad
\forall t > 0.
\end{equation}

\noindent
We divide by $t$ and we get

\begin{equation}
\label{g_prime_estimate_6}
t^{a-b-2}
\{-b 8 \varepsilon + a - b\} 
\leq
\frac{g^{\prime}(t)}{t}
\leq
[t^{a-b-2} + t^{a+b-2}]
\{b 8 \varepsilon + a + b\}
\qquad
\forall t > 0.
\end{equation}

\noindent
We need to estimate $g^{\prime\prime}(t) t$; to this aim, we use (\ref{g^{prime_prime}}):

\begin{eqnarray}
\label{g^{prime_prime}(t)t}
g^{\prime\prime}(t)t = 
\left[g^{\prime}(t)  - \frac{g(t)}{t}\right]
\left\{
[b \cos(\varphi(t))] \varphi^{\prime}(t) t \ln t + a+b\sin(\varphi(t))
\right\} +
\phantom{aaaaa}
\nonumber
\\
\frac{g(t)}{t}
\left\{
[-b \sin(\varphi(t))] \varphi^{\prime}(t) t \varphi^{\prime}(t) t \ln t + 
\phantom{aaaaaaaaaaaaaaaaaaaaaaa}
\right.
\nonumber
\\
\left.
[b \cos(\varphi(t))] [\varphi^{\prime\prime}(t) t^2 \ln t + \varphi^{\prime}(t) t (\ln t + 2)] 
\right\}.
\end{eqnarray}

\noindent
We keep in mind (\ref{g^prime}) and we can write as follows

\begin{eqnarray}
\label{g^{prime_prime}(t)t_bis}
g^{\prime\prime}(t)t = 
g^{\prime}(t)
\left\{
[b \cos(\varphi(t))] \varphi^{\prime}(t) t \ln t + a-1+b\sin(\varphi(t))
\right\} +
\phantom{aaaaa}
\nonumber
\\
\frac{g(t)}{t}
\left\{
[-b \sin(\varphi(t))] \varphi^{\prime}(t) t \varphi^{\prime}(t) t \ln t + 
\phantom{aaaaaaaaaaaaaaaaaaaaaaa}
\right.
\nonumber
\\
\left.
[b \cos(\varphi(t))] [\varphi^{\prime\prime}(t) t^2 \ln t + \varphi^{\prime}(t) t (\ln t + 2)] 
\right\}.
\end{eqnarray}
For simplicity, define
\begin{flalign*}
&\Phi_{1}(t):=\left[b\cos(\varphi(t))\varphi'(t)t\ln t+a-1+b\sin(\varphi(t))\right]\\
&\Phi_{2}(t):=-b \sin(\varphi(t)) \varphi^{\prime}(t) t \varphi^{\prime}(t) t \ln t + [b \cos(\varphi(t))] [\varphi^{\prime\prime}(t) t^2 \ln t + \varphi^{\prime}(t) t (\ln t + 2)],
\end{flalign*}
in such a way that \eqref{g^{prime_prime}(t)t_bis} reads as
\begin{equation}
\label{g''(t)t}
g''(t)t=g'(t)\Phi_{1}(t)+\frac{g(t)}{t}\Phi_{2}(t).
\end{equation}
By \eqref{inequality_for_varphi_prime_1}, \eqref{inequality_for_varphi_prime_2} and \eqref{inequality_for_varphi_prime_prime} we estimate for $t>1$
\begin{flalign*} 
-8\varepsilon b + a - 1 - b
\leq 
\Phi_{1}(t) \leq 
&|\Phi_{1}(t)|\le 8\varepsilon b + a - 1 + b;
\\
&|\Phi_{2}(t)|\le b 8\varepsilon 8\varepsilon + b [ 128 \varepsilon + 8 \varepsilon + 16 \varepsilon] = b \varepsilon [ 64 \varepsilon + 152].
\end{flalign*}
Now we estimate $g^{\prime\prime}(t)t$ from below; when $t>1$ we keep in mind positivity of $g^{\prime}$, $g$ and estimates for $\Phi_{1}$, $\Phi_{2}$: we have 
$$
g''(t)t=g'(t)\Phi_{1}(t)+\frac{g(t)}{t}\Phi_{2}(t)
\geq
g'(t)\{ -8\varepsilon b + a - 1 - b \} + \frac{g(t)}{t} ( - b) \varepsilon [ 64 \varepsilon + 152 ] =: \cd{\mbox{(I)}};  
$$
now we use the right hand side of (\ref{g_prime_estimate_3}) and we get
$$
\cd{\mbox{(I)}}
\geq
g'(t) \left\{ -8\varepsilon b + a - 1 - b  + \frac{  - b \varepsilon [ 64 \varepsilon + 152 ]}{-8\varepsilon b + a - b} \right\}
=:\cd{\mbox{(II)}};  
$$
%\noindent
now we use (\ref{epsilon_1}): $\varepsilon < 1$ gives $64\varepsilon + 152 < 216$ and 
$\varepsilon < \frac{a-1-b}{224b}$ gives $1<-b8\varepsilon + a -b$, so that

\begin{equation}
\label{216}
\frac{b \varepsilon \{ 64 \varepsilon + 152  \}}{-b8\varepsilon + a - b} < 216 \, b \, \varepsilon;
\end{equation} 
then
$$
\cd{\mbox{(II)}}
\geq
g'(t) \left\{ -8\varepsilon b + a - 1 - b   - 216 \varepsilon b \right\};  
$$
this means that, for $t>1$ we have
$$
g''(t)t
\geq
g'(t) \left\{ - 224 \varepsilon b + a - 1 - b    \right\}.  
$$
Note that 
we required  $
%\boxed{
-224 b \varepsilon + a-1-b
>0
%}
$ in our assumption (\ref{epsilon_1}).

\noindent
When $t \in (0, 1]$, we have $\varphi(t) = \frac{3}{2} \pi$, $\varphi^{\prime}(t) = 0 = \varphi^{\prime\prime}(t)$; then
$g^{\prime\prime}(t)t = g^{\prime}(t) (a-1-b)$. Moreover, $g'$ is positive and

\begin{eqnarray}
\label{comparison_2}
a-1-b>a-1-b- 224 b \varepsilon >0,
\end{eqnarray}
 
\noindent
then, 

\begin{equation}
\label{g''(t)t_bis}
g^{\prime\prime}(t)t \geq 
g'(t) \left\{ - 224 \varepsilon b + a - 1 - b    \right\}
\qquad
\forall t>0.
\end{equation}

\noindent
Since $g'(t)>0$ when $t>0$, this last inequality guarantees that $g^{\prime\prime}(t)>0$ for all $t > 0$; then $g^{\prime}$ strictly increases in $(0,+\infty)$; since $g^{\prime}$ is continuous in $[0, +\infty)$, then  $g^{\prime}$ strictly increases in $[0,+\infty)$: this guarantees that $g$ is strictly convex in $[0, +\infty)$.

\noindent
Now we estimate $g^{\prime\prime}(t)t$ from above; when $t>1$ we keep in mind positivity of $g^{\prime}$, $g$ and estimates for $\Phi_{1}$, $\Phi_{2}$: we have 
$$
g''(t)t=g'(t)\Phi_{1}(t)+\frac{g(t)}{t}\Phi_{2}(t)
\leq
g'(t)\{ 8\varepsilon b + a - 1 + b \} + \frac{g(t)}{t} b \varepsilon [ 64 \varepsilon + 152 ] =: \cd{\mbox{(III)}};  
$$
now we use the right hand side of (\ref{g_prime_estimate_3}) and we get
$$
\cd{\mbox{(III)}}
\leq
g'(t) \left\{ 8\varepsilon b + a - 1 + b  + \frac{   b \varepsilon [ 64 \varepsilon + 152 ]}{-8\varepsilon b + a - b} \right\}
=:\cd{\mbox{(IV)}};  
$$
%\noindent
we use (\ref{216}) and we get
$$
\cd{\mbox{(IV)}}
\leq
g'(t) \left\{ 8\varepsilon b + a - 1 + b   + 216 \varepsilon b \right\};  
$$
this means that, for $t>1$ we have
$$
g''(t)t
\leq
g'(t) \left\{  224 \varepsilon b + a - 1 + b    \right\}.  
$$

\noindent
When $t \in (0, 1]$, we have $\varphi(t) = \frac{3}{2} \pi$, $\varphi^{\prime}(t) = 0 = \varphi^{\prime\prime}(t)$; then
$g^{\prime\prime}(t)t = g^{\prime}(t) (a-1-b)$. Moreover, $g'$ is positive so that

\begin{equation}
\label{g''(t)t_tris}
g^{\prime\prime}(t)t \leq 
g'(t) \left\{  224 \varepsilon b + a - 1 + b    \right\}
\qquad
\forall t>0.
\end{equation}

\noindent
This ends the proof of Theorem \ref{thm_g}.
\qed

\section{Another example}
\label{example2}

\noindent
Now we give an example in the subquadratic case by modifing a little bit the previous example of section \ref{example}:
we introduce an additional restriction on $a$, $b$ and we select a smaller $\varepsilon$. More precisely,  
We fix $a,b \in (0, +\infty)$ with (\ref{conditions_for_a_b}) as in section \ref{example}; moreover, we require, in addition,
\begin{equation}
\label{conditions_for_a_b_subquadratic}
a+b<2.
\end{equation}
\noindent
We consider $g:[0, +\infty) \to [0, +\infty)$ given by (\ref{definition_for_g}) with $\varphi$ as in (\ref{definition_for_varphi}) with 
$\varepsilon>0$ satisfing (\ref{epsilon_1}) as in section \ref{example}; moreover, we require, in addition,  

\begin{equation}
\label{epsilon_subquadratic}
%0 < 
\varepsilon < 
%\min
%\left\{
%1;
%\,\, 
%\frac{a-1-b}{224\,\,b}
%;
%\,\,
\frac{2-a-b}{224\,\,b}
%\right\}
.
\end{equation}
Please, note that \eqref{conditions_for_a_b_subquadratic} gives $0 < 2-a-b$, so the requirement \eqref{epsilon_subquadratic} is in accordance 
%agrees 
with $0<\varepsilon$ and it implies 
$$224 b\varepsilon + a - 2 + b<0.$$
\noindent
This and the right hand side of \eqref{statement_3} 
%results 
in Theorem \ref{thm_g} give

\begin{thrm}
\label{thm_g_subquadratic}
Let us consider $a,b \in (0, +\infty)$ verifing (\ref{conditions_for_a_b}), (\ref{conditions_for_a_b_subquadratic}); we consider $g(t)$ given by (\ref{definition_for_g}) where $\varphi$ is defined in (\ref{definition_for_varphi}) and $\varepsilon$ satisfies (\ref{epsilon_1}), (\ref{epsilon_subquadratic}). Then 
$$
 g^{\prime\prime}(t) - \frac{g^{\prime}(t)}{t} \leq \{ 224 b\varepsilon + a - 2 + b \} \frac{g^{\prime}(t)}{t} < 0
\qquad \forall t > 0
$$
and we get $M=1$ in the right hand side of \eqref{assumption_BM}. 
Since
$$
\left(
\frac{g'(t)}{t}
\right)^{\prime}
=
\frac{g''(t)t - g'(t)}{t^2} 
=
\left(
g''(t) - \frac{g'(t)}{t}
\right) \frac{1}{t} < 0,
$$
we get
$$
t \to \frac{g'(t)}{t} \text{ strictly decreases in } (0, +\infty).
$$

%Let $g(t) = t^{a+b\sin(\varphi(t))}$ with $\varphi$ as before.  
%Then $g$ verifies all the assumptions of Theorem \ref{thm_LMS2003}; more precisely, (\ref{1.2_LMS_2003}) holds true with $p=a-b$, $q=a+b$, $\Lambda_1=-b8\varepsilon + a - b$, $\Lambda_2=b8\varepsilon + a + b$ and (\ref{senza_nome_LMS_2003}) is satisfied with $\gamma = \frac{1}{a-1-b-224b\varepsilon}$.
\end{thrm}

%\noindent
%Note that (\ref{epsilon_1}) implies $\Lambda_1=-b8\varepsilon + a - b > 1$ and

%\begin{eqnarray}
%\label{comparison_1}
%1>1-b>a-1-b>a-1-b-224 b \varepsilon >0,
%\end{eqnarray}
 
%\noindent
%then, 

%\begin{eqnarray}
%\label{gamma>1}
%\gamma = 
%\frac{1}
%{a-1-b-224b\varepsilon } >1
%\end{eqnarray}

\end{document}